\documentclass[reqno]{amsart}  

\usepackage{amsmath,amsfonts,amssymb,amsthm,latexsym}
\usepackage{mathtools}
\usepackage[T1]{fontenc}
\usepackage[latin1]{inputenc}
\usepackage[USenglish]{babel}

\usepackage{graphicx} 
\usepackage{epstopdf}

\usepackage{tikz}
\usetikzlibrary{matrix,arrows,cd}
\usepackage{stix}
\usepackage{hyperref}

\newtheorem{thm}{Theorem}[section]
\newtheorem{prop}[thm]{Proposition}
\newtheorem{lem}[thm]{Lemma}
\newtheorem{cor}[thm]{Corollary}
\theoremstyle{definition}
\newtheorem{defi}[thm]{Definition} 
\newtheorem{example}[thm]{Example} 

\newtheorem{remark}{Remark}

\numberwithin{equation}{section}

\DeclareMathOperator{\Id}{Id}

 \DeclareMathOperator{\Ad}{Ad}

  \DeclareMathOperator{\im}{Im}
  \DeclareMathOperator{\st}{st}
 \DeclareMathOperator{\Z}{Z}
 \DeclareMathOperator{\coker}{coker}

\begin{document}

\title{Universal central extensions of braided crossed modules in Lie algebras}
	\author{A. Fern\'andez-Fari\~na}
	\address{[A. Fern\'andez-Fari\~na] Departamento de Matem\'aticas, Universidade de Santiago de Compostela, 15782, Spain.}
\email{alejandrofernandez.farina@usc.es}

	\author{M. Ladra}
\address{[M. Ladra] Departamento de Matem\'aticas, Instituto de Matem\'aticas, Universidade de Santiago de Compostela, 15782, Spain.}
\email{manuel.ladra@usc.es}

\thanks{This work was partially supported by Agencia Estatal de Investigaci\'on (Spain), grant MTM2016-79661-P (European FEDER
support included, UE) and by Xunta de Galicia through the Competitive Reference Groups (GRC), ED431C 2019/10. The first author is also supported by a scholarship of Xunta de Galicia (Spain), grant ED481A-2017/064.}

\begin{abstract}
 In this paper, we give a natural braiding on the universal central extension of a crossed module of Lie algebras with a given braiding and construct the universal central extension of a braided crossed module of Lie algebras, showing that, when one of the constructions exists, both exist and coincide.	
\end{abstract}


\maketitle

\section*{Introduction}
In \cite{Fuk} Fukushi did a braided version of the results of universal central extensions of crossed modules of groups given by Norrie in \cite{Norr}. He did in the sense that he found a natural braiding on the universal central extension of crossed modules of groups which works well with one braided crossed module given. But it is not the universal central extension of braided crossed modules since we need some restrictions in the notion of center and commutator when we add the braidings.

In this paper we will do a braided version of the results given by Casas and Ladra in \cite{PCMLie} for braided crossed modules of Lie $K$-algebras in the sense that we will study the universal braided crossed module. For that we will need the definition of center and commutator given by Huq in \cite{Huq} in the braided category.

\section{Preliminaries}

\begin{defi}
	A \emph{crossed module of Lie algebras} is a pair $(M\xrightarrow{\partial}N,\cdot)$ where:
	
	\begin{itemize}
		\item $\cdot$ is Lie left-action of $N$ on $M$, i.e. a $K$-bilinear map $\cdot\colon N\times M\rightarrow M$, $(n,m)\mapsto n\cdot m$, satisfying for all $m,m'\in M$, $n,n'\in N$:
		\begin{align*}
		[n,n']\cdot m & =n\cdot(n'\cdot m)-n'\cdot(n\cdot m),\\
		n\cdot[m,m'] & =[n\cdot m,m']+[m,n\cdot m'], 
		\end{align*}
		
		\item $\partial\colon M\xrightarrow{}N$ is a Lie $K$-homomorphism which satisfies the following properties:
		\begin{itemize}
			\item
			$\partial$ is $N$-equivariant Lie $K$-homomorphism:
\begin{equation*}
			\partial(n\cdot m)=\Ad(n)(\partial(m))=[n,\partial(m)],
			\end{equation*}
			\item
			$\partial$ satisfies the Peiffer identity:
			\begin{equation*}
			\partial(m)\cdot m'=\Ad(m)(m')=[m,m'],
			\end{equation*}
		\end{itemize}
		
	\end{itemize}
\end{defi}

\begin{example}
	The identity map $M\xrightarrow{\Id_M}M$ with the adjoint action, $m\cdot m'=[m,m']$ defines a crossed module.
\end{example}

\begin{defi}
	Let $(M\xrightarrow{\partial} N, \cdot)$ and $(L\xrightarrow{\delta}H,*)$ be two crossed modules of Lie algebras.
	
	A \emph{morphism} between that crossed modules is a pair of Lie $K$-homomorphisms $(f_1,f_2)$, $f_1\colon M \xrightarrow{}L$ and $f_2\colon N\xrightarrow{}H$, such that:
	\begin{align}
	f_1(n\cdot m)& =f_2(n)*f_1(m),  \quad  \text{for all} \ \  m\in M, n\in N,	 \label{XLieH1} \tag{XLieH1} \\
	\delta \circ f_1&=f_2\circ \partial. \label{XLieH2} \tag{XLieH2}
	\end{align}

\end{defi}

The category of crossed modules of Lie $K$-algebras is a semi-abelian category in the sense of \cite{JLT} and it will be denoted by
$\textbf{\textit{X}}(\textbf{\textit{LieAlg}}_K)$.

The following definitions will be also necessary for the development of the paper.

\begin{defi}
 An \emph{extension} in $\textbf{\textit{X}}(\textbf{\textit{LieAlg}}_K)
 $ is a regular epimorphism, i.e.   a surjective morphism.

  Following the theory in~\cite{JaKe}, we have three kinds of extensions: trivial, normal and central.

  An extension $\Phi \colon (X_1\xrightarrow{\rho}X_2,*)\longrightarrow (M\xrightarrow{\partial} N,\cdot)$ is \emph{trivial} if the induced square
\begin{center}
	\begin{tikzcd}(X_1\xrightarrow{\rho}X_2,*) \arrow[r,"\Phi"]\arrow[d] & (M\xrightarrow{\partial} N,\cdot)\arrow[d]\\ (X_1\xrightarrow{\rho}X_2,*)_{ab} \arrow[r,"\Phi_{ab}"] & (M\xrightarrow{\partial} N,\cdot)_{ab}
	\end{tikzcd}
\end{center}
is a pullback in $\textbf{\textit{X}}(\textbf{\textit{LieAlg}}_K)$.

An extension is \emph{normal} if one of the projections of the kernel pair is trivial.

An extension is \emph{central} if there exists another extension $\Psi\colon (Y_1\xrightarrow{\varphi}Y_2,\star)\longrightarrow(M\xrightarrow{\partial}N,\cdot)$ such that $\pi_2$ (also denoted $\Psi^*(\Phi)$) in the pullback
\begin{center}
	\begin{tikzcd}(X_1\xrightarrow{\rho}X_2,*)\times_{(M\xrightarrow{\partial}N,\cdot)}(Y_1\xrightarrow{\varrho}Y_2,\star) \arrow[r,"\pi_1"]\arrow[d,"\pi_2"] & (X_1\xrightarrow{\rho}X_2,*)\arrow[d,"\Phi"]\\ (Y_1\xrightarrow{\varrho}Y_2,\star) \arrow[r,"\Psi"] & (M\xrightarrow{\partial} N,\cdot),
	\end{tikzcd}
\end{center}
 is trivial.

The central extensions in $\textbf{\textit{X}}(\textbf{\textit{Lie}}_K)$ over an object $(M\xrightarrow{\partial}N,\cdot)$ conform another category, whose morphisms are the morphisms of crossed modules $\Theta\colon (X_1\xrightarrow{\rho}X_2,*)\longrightarrow (Y_1\xrightarrow{\varrho}Y_2,\star)$ making commutative the diagram
\begin{center}
	\begin{tikzcd}
	(X_1\xrightarrow{\rho}X_2,*)\arrow[rr,"\Theta"]\arrow[dr,"\Phi"'] & & (Y_1\xrightarrow{\varrho}Y_2,\star) \arrow[dl,"\Psi"] \\ &(M\xrightarrow{\partial}N,\cdot).
	\end{tikzcd}
\end{center}

A central extension $\mathcal{U}$ is said to be \emph{universal} (over
$(M\xrightarrow{\partial}N,\cdot)$) if it is the initial object in the category of central extensions over $(M\xrightarrow{\partial}N,\cdot)$. From the definition, it is clear that the universal central extension is unique up to isomorphisms.
\end{defi}

The notion of center of a object  was defined in \cite{Huq}, in a category with certain properties. This construction only needs that the category has finite products and zero object.


The category $\textbf{\textit{X}}(\textbf{\textit{Lie}}_K)$ has center in the sense of Huq \cite{Huq}, and it was defined in \cite{PCMLie}.

\begin{defi}
	The \emph{center} of a crossed module of Lie $K$-algebras $(M\xrightarrow{\partial}N,\cdot)$ is the crossed submodule $Z((M\xrightarrow{\partial}N,\cdot))=(M^N\xrightarrow{\partial|_{M^N}}\st_{N}(M)\cap Z(N),\cdot_Z)$, where:
	\begin{itemize}
		\item $Z(N)=\{n\in N\mid [n,n']=0, \ n'\in N \}$ is the center of the Lie $K$-algebra $N$,
		\item $M^N=\{m\in M\mid n\cdot m=0, \ n\in N\}$,
		\item $\st_N(M)=\{n\in N\mid n\cdot m=0, \  m\in M \}$,
	\end{itemize}
	and $\cdot_Z$ is the induced action, which means that is the zero action by the definition of $M^N$.
\end{defi}

In our semi-abelian context, the concepts of normal and central extension are equivalent, and a more practical characterization is the following:

	An extension $(M\xrightarrow{\partial}N)\xrightarrow{f=(f_1,f_2)}(L\xrightarrow{\delta}H,*)$ is central if and only if
$\ker(f)=(\ker(f_1)\xrightarrow{\partial|_{\ker(f_1)}}\ker(f_2),\cdot_{\ker})$ is a crossed submodule of the center, i.e. of the crossed submodule $(M^N\xrightarrow{\partial|_{M^N}}\st_{N}(M)\cap Z(N),\cdot_Z)$.

The notions of commutator of a crossed module and of a perfect crossed module were introduced in \cite{PCMLie}. This notion of commutator coincides in this category with the idea of commutator given by Huq in \cite{Huq} in a category with products, zero objects, kernels and cokernels.

If $L$ is a Lie $K$-algebra and $S\subset L$, we denote $\langle S\rangle_L$ the Lie subalgebra of $L$ generated by $S$, that means, the intersection of all subalgebras that contains $S$.

\begin{defi}
	Let $(M\xrightarrow{\partial}N,\cdot)$ be a crossed module of Lie $K$-algebras. The \emph{commutator crossed submodule} is $(D_N(M)\xrightarrow{\partial|_{D_N(M)}}[N,N],\cdot_C)$ where $\cdot_C$ is the induced action and:
	\begin{itemize}
\item $D_N(M)=\langle\{x\in M \mid \text{there exist} \ m\in M ,  n\in N  \text{ such that} \ x=n\cdot m \}\rangle_{M}$
		\item $[N,N]=\langle\{y\in N\mid  \text{there exist} \  n,n'\in N  \text{ such that} \  y=[n,n']\}\rangle_N $ is the commutator of the Lie $K$-algebra $N$.
	\end{itemize}
	
\end{defi}

\begin{defi}
	We will say that a crossed module Lie $K$-algebras $(M\xrightarrow{\partial}N,\cdot)$ is a perfect crossed module Lie $K$-algebras if it coincides with its commutator crossed submodule, i.e. $M=D_N(M)$ and $N=[N,N]$.
\end{defi}

The universal central extension is totally related with the concept non-abelian tensor product. In the case of groups, Brown and Loday  in \cite{BrLod} define the non-abelian tensor product of groups and proved that the universal central extension is the non-abelian tensor product with the epimorphism that takes generators to its commutator. The same happen in case of Lie algebras with the non-abelian tensor product of Lie algebras introduced by Ellis in \cite{Ellis}.

In the cases of crossed modules of groups \cite{Norr} and crossed modules of Lie algebras \cite{PCMLie}, is needed too the notion of non-abelian tensor product.

For the case of non-abelian tensor product of Lie $K$-algebras, the definition introduced by Ellis in \cite{Ellis} is the following one:
\begin{defi}
	Let $M$ and $N$ be two Lie $K$-algebras such that $M$ acts in $N$ by $\cdot$ and $N$ acts in $M$ with $*$.
	
	The \emph{non-abelian tensor product}, denoted by  $M\otimes N$, is defined as the Lie $K$-algebra generated by the symbols $m\otimes n$ with $m\in M$, $n\in N$ and the relations:
	\begin{align}
	\lambda(m\otimes n)&=\lambda m\otimes n=m\otimes \lambda n, \label{RTLie1}\tag{T1}\\
	(m+m')\otimes n&=m\otimes n+m'\otimes n, \label{RTLie2} \tag{T2} \\
	m\otimes (n+n')&=m\otimes n+m\otimes n',  \notag \\
	[m,m']\otimes n&=m\otimes (m'\cdot n)-m'\otimes (m\cdot n), \label{RTLie3} \tag{T3} \\
	m\otimes [n,n']&=(n'*m)\otimes n-(n*m)\otimes n',\notag \\
	[(m\otimes n),(m'\otimes n')]&=-(n*m)\otimes(m'\cdot n'),\label{RTLie4} \tag{T4}
	\end{align}
	where $m,m'\in M$, $n,n'\in N$, $\lambda  \in K$.
\end{defi}

	When we talk about $M\otimes M$ we will assume that $M$ acts on itself by the adjoint action.

Now we will consider braidings in crossed modules of Lie algebras, whose definitions were introduced in \cite{FFBraidI,Ulua}  (see also  \cite{FFBraidII}).

\begin{defi}
	Let $(M\xrightarrow{\partial}N,\cdot)$ be a crossed module of Lie $K$-algebras.
	
	A \emph{braiding} (or \emph{Peiffer lifting}) on the crossed module is a $K$-bilinear map
	$\{-,-\}\colon N\times N\xrightarrow{ } M$ satisfying:
	\begin{align}
	\partial\{n,n'\}&=[n,n'], \label{BLie1}\tag{BLie1} \\
	\{\partial m, \partial m' \}&=[m,m'], \label{BLie2} \tag{BLie2}\\ 
	\{\partial m, n \}&=-n\cdot m, \label{BLie3} \tag{BLie3} \\	
	\{n,\partial m \}&=n\cdot m,  \label{BLie4} \tag{BLie4} \\
	\{n,[n',n'']\}&=\{[n,n'],n''\}-\{[n,n''],n'\}, \label{BLie5}\tag{BLie5}\\
    \{[n,n'],n''\}&=\{n,[n',n'']\}-\{n',[n,n'']\}, \label{BLie6}\tag{BLie6}
	\end{align}
	for $m,m'\in M$, $n,n',n''\in N$.
	
	If $\{-,-\}$ is a braiding on $\mathcal{X}$ we will say that $(M\xrightarrow{\partial}N,\cdot,\{-,-\})$ is a braided crossed module of Lie $K$-algebras.
\end{defi}

\begin{example}\label{ExBra} \hfill
	\begin{enumerate}
	\item There is a canonical braidings on $(M\xrightarrow{\Id_M}M,[-,-])$, given by 	
	\[M\times M\xrightarrow{}M, \ (m,m')\mapsto [m,m'].\]
	
	\item There is a canonical braiding $(M\otimes M\xrightarrow{\partial} M,\cdot)$, $m\otimes m'\mapsto [m,m']$ with $m\cdot (m_1\otimes m_2)=m\otimes [m_1,m_2]$. It is given by
	\[M\times M\xrightarrow{}M\otimes M, \ (m,m')\mapsto m\otimes m'.\]
	\end{enumerate}
\end{example}

\begin{defi}
	A \emph{morphism of braided crossed modules of Lie $K$-algebras}, denoted as $(M\xrightarrow{\partial}N,\cdot,\partial,\{-,-\})\xrightarrow{(f_1,f_2)}(L\xrightarrow{\delta}H,*,\lParen-,-\rParen)$, is a homomorphism of crossed modules of Lie $K$-algebras such that:
\begin{equation}\label{BXLieH3}
	f_1(\{n,n'\})=\lParen f_2(n),f_2(n')\rParen, \qquad \text{for}   \ \ n,n'\in N. \tag{BXLieH3}
	\end{equation} 
\end{defi}


In the case of the braiding category, the concept of braiding changes a little the concept of center and commutator from the category of crossed modules, appearing the following subobjects using the definition given by Huq \cite{Huq} in the general case.

\begin{defi}
	The \emph{center} of a braided crossed module $(M\xrightarrow{\partial}N,\cdot,\{-,-\})$ is the braided crossed submodule $Z((M\xrightarrow{\partial}N,\cdot))=(M^N\xrightarrow{\partial|_{M^N}}\Z_B(N),\cdot_Z,\{-,-\}_Z)$, where:
	\begin{equation*}
		\Z_B(N)=\{n\in N\mid \{n,n'\}=0=\{n',n\}, \  n'\in N \}.
	\end{equation*}
	$\cdot_Z$ is the induced action and $\{-,-\}_Z$ the induced braided, which means that is the zero action and the zero braiding by the definition of $M^N$ and $\Z_B(N)$.
\end{defi}

\begin{remark}
	It is easy to show that the following inclusions of subalgebras are true:
	\[ M^N\subset Z(M), \qquad  Z_B(N)\subset Z(N)\cap \st_N(M).\]
	
	In addition, if we use the properties \eqref{BLie3} and \eqref{BLie4}, then we have that $M^N=\{m\in M\mid \partial(m)\in Z_B(N)\}$.
\end{remark}

Imitating the case of crossed modules, we have the following definition:

\begin{defi}
	An \emph{extension of braided crossed modules} of Lie $K$-algebras is a morphism $(M\xrightarrow{\partial}N,\cdot,\{-,-\})\xrightarrow{(f_1,f_2)}(L\xrightarrow{\delta}H,*,\lParen-,-\rParen)$
	such that $f_1$ and $f_2$ are surjective morphisms.
	
	Besides, we will say that it is \emph{central} if $(\ker(f_1)\xrightarrow{\partial|_{\ker(f_1)}}\ker(f_2),\cdot_{\ker},\{-,-\}_{\ker})$ is a braided crossed submodule of $(M^N\xrightarrow{\partial|_{M^N}}Z_B(N),\cdot_Z,\{-,-\}_Z)$, that is, the kernel is ``inside'' the center.
\end{defi}

\begin{defi}
		We will say that a extension of braided crossed modules of Lie $K$-algebras is a \emph{compatible central extension} if it is central as a crossed module extension.
\end{defi}

It is immediate using inclusions that every central extension in the category of braided crossed modules of Lie algebras is a compatible central extension. The next example shows that not every compatible central extension is a central extension.

\begin{example}
	We will take $M$ as an abelian $K$-Lie algebra of finite dimension $n$, i.e., $M$ is isomorphic as vectorial space to $K^n$ with the Lie bracket $[x,y]=0$ for $x,y\in M$.
	
	Using the second example in Example~\ref{ExBra}, we have that $(M\otimes M\xrightarrow{\partial} M,\cdot)$, with $\partial=0$, $m\cdot (m_1\otimes m_2)=m\otimes [m_1,m_2]=m\otimes 0=0$ and $\{m,m'\}=m\otimes m$ is a braided crossed module of Lie $K$-algebras.

	Note that, since $M$ is abelian, we have that $M\otimes M$ is isomorphic as vector space to the usual tensor product and its Lie bracket is $0$.

	In particular, we have for $M=K^2$ and $M=K^3$ the braided crossed modules $(K^2\otimes K^2\xrightarrow{0}K^2,0,-\otimes-)$ and $(K^3\otimes K^3\xrightarrow{0}K^3,0,-\otimes-)$, where we can think  the tensor product as the usual one.
	
	We will take the projection $K^3\xrightarrow{\pi}K^2$, $(x,y,z)\mapsto (x,y)$. Sice $\pi$ is surjective, is easy to show that $\pi\otimes \pi\colon K^3\otimes K^3\xrightarrow{}K^2\otimes K^2$ is surjective and
	\[
	 (K^3\xrightarrow{0}K^3,0,-\otimes-)\xrightarrow{ \ (\pi\otimes\pi,\pi) \ } (K^2\otimes K^2\xrightarrow{0}K^2,0,-\otimes-),
	\]
	is an extension of braided crossed modules of Lie $K$-algebras.
	
	We will show that this extension is a compatible central extension that is not a central extension. Is easy to show that the correspondent subalgebras are the following ones:
	\begin{itemize}
		\item $(K^3\otimes K^3)^{K^3}=K^3\otimes K^3$,
		\item $\st_{K^3}(K^3\otimes K^3)=K^3$,
		\item $Z(K^3)=K^3$.
		\item $Z_B(K^3)=0$ (since it is the usual tensor product).
	\end{itemize}
	It is immediate that $\ker(\pi\otimes \pi)\subset (K^3\otimes K^3)^{K^3}=K^3\otimes K^3$ and $\ker(\pi)\subset Z(K^3)\cap \st_{K^3}(K^3\otimes K^3)=K^3$, so the extension is a compatible central extension.
	
	However $0\neq\ker(\pi)=\{(x,y,z)\in K^3\mid y=z=0\}\nsubseteq Z_B(K^3)=0$, so the extension is not central.
\end{example}

\begin{defi}
	Let $(M\xrightarrow{\partial}N,\cdot,\{-,-\})$ be a braided crossed module of Lie $K$-algebras. The commutator braided crossed submodule is given by $(B_N(M)\xrightarrow{\partial|_{B_N(M)}}[N,N],\cdot_C,\{-,-\}_C)$ where $\cdot_C$ and $\{-,-\}_C$ are the induced operations and
	\begin{equation*}
		B_N(M)=\big\langle \big \{x\in M\mid \text{there exist} \ n,n'\in N \ \text{such that} \ x=\{n,n'\} \big\}\big\rangle_{M}.
	\end{equation*}
\end{defi}

\begin{remark}
	$B_N(M)$ is a ideal of $M$ and we have the following inclusion of subalgebras:
	$$[M,M]\subset D_{N}(M)\subset B_N(M).$$
\end{remark}

\begin{defi}
	We will say that a braided crossed module Lie $K$-algebras $(M\xrightarrow{\partial}N,\cdot,\{-,-\})$ is a perfect braided crossed module Lie $K$-algebras if it coincides with its commutator braided crossed submodule, that is $M=B_N(M)$ and $N=[N,N]$.
\end{defi}

\section{The universal central extension for perfect braided crossed modules of Lie Algebras}
	
As in the case for crossed modules of Lie $K$-algebras we have the following definition for the universal central extension in the case of braided crossed modules of Lie algebras.
\begin{defi}
	We will say that $\mathcal{U}\xrightarrow{u}\mathcal{X}$ is the universal central extension of $\mathcal{X}$ if and only if, when $\mathcal{Z}\xrightarrow{f}\mathcal{X}$ is another central extension of braided crossed modules of $K$-algebras, exists a unique morphism $h^X\colon \mathcal{U}\xrightarrow{}\mathcal{Z}$ such that $u=f\circ h^X$.
\end{defi}
	
In this section we will find the expression of this initial universal object when it exists and we will try to characterize this fact. 	
	
\begin{lem}\label{braidmorph}
	If $(M\xrightarrow{\partial}N,\cdot, \{-,-\})$ is a braided crossed module of Lie algebras, then $N\otimes N\xrightarrow{\Phi_1} M$ defined by $n\otimes n'\mapsto \{n,n'\}$ and $N\otimes N\xrightarrow{\Phi_2} N$ defined by $n\otimes n'\mapsto [n,n']$ are Lie $K$-homomorphisms.
	
	Besides $\Phi_1$ and $\Phi_2$ are simultaneously surjective if and only if the crossed module $(M\xrightarrow{\partial}N,\cdot, \{-,-\})$ is perfect.
\end{lem}
	\begin{proof}
		Since $\Phi_1$ and $\Phi_2$ are defined by generators, we only need to prove that they are well defined to prove that they are Lie $K$-homomorphisms.
		
		First we will prove that the two morphisms preserve the relations.
		
		Since $[-,-]$ and $\{-,-\}$ are $K$-bilinear is immediate that \eqref{RTLie1} and \eqref{RTLie2} are preserved.
		
		For that, since the two actions given to make  $N\otimes N$ are the Lie bracket, $[-,-]$, of $N$, we can rewrite the relations \eqref{RTLie3} and \eqref{RTLie4} to obtain the following ones:
		\begin{align*}
		[n_1,n_2]\otimes n_3& =n_1\otimes [n_2,n_3]-n_2\otimes [n_1,n_3],  \qquad   \qquad \qquad \text{(T3)}\\
		n_1\otimes [n_2,n_3]&=[n_3,n_1]\otimes n_2-[n_2,n_1]\otimes n_3,   \\
		[(n_1\otimes n_2),(n_3\otimes n_4)]&=[n_1,n_2]\otimes[n_3, n_4]. \qquad  \qquad   \qquad  \qquad \qquad \ \ \text{(T4)}
		\end{align*}

		Starting with \eqref{RTLie3} we have:
		\begin{align*}
		&\Phi_1([n_1,n_2]\otimes n_3)=\{[n_1,n_2],n_3\}=\{n_1,[n_2,n_3]\}-\{n_2,[n_1,n_3]\}\\&=\Phi_1(n_1\otimes [n_2,n_3])-\Phi_1(n_2\otimes [n_1,n_3])=\Phi_1(n_1\otimes [n_2,n_3]-n_2\otimes [n_1,n_3]),
		\end{align*}
		where we use \eqref{BLie6}.
		
		We will see now the second relation in \eqref{RTLie3}:
		\begin{align*}
		&\Phi_1(n_1\otimes [n_2,n_3])=\{n_1,[n_2,n_3]\}=\{[n_1,n_2],n_2\}-\{[n_1,n_3],n_2\}\\&=\Phi_1([n_1,n_2]\otimes n_3-[n_1,n_3]\otimes n_2)=\Phi_1(-[n_2,n_1]\otimes n_3+[n_3,n_1]\otimes n_2)\\&=\Phi_1([n_3,n_1]\otimes n_2-[n_2,n_1]\otimes n_3),
		\end{align*}
		where we use \eqref{BLie5}.
		
		For $\Phi_2$ is true using a similar argument together the Jacobi identity in both equalities.
		
		The proof for \eqref{RTLie4} is immediate for $\Phi_2$, since it is a Lie Morphism and both equalities are $[[n_1,n_2],[n_3,n_4]]$ after applying $\Phi_2$.
		
		For $\Phi_1$ we have the following equalities:
		\begin{align*}
		&\Phi_1([n_1\otimes n_2,n_3\otimes n_4])=[\Phi_1(n_1\otimes n_2),\Phi_1(n_3\otimes n_4)]=[\{n_1,n_2\},\{n_3,n_4\}]\\
		&=\{\partial\{n_1,n_2\},\partial\{n_3,n_4\}\}=\{[n_1,n_2],[n_3,n_4]\}=\Phi_1([n_1,n_2]\otimes[n_3,n_4]),
		\end{align*}
		where we use \eqref{BLie2} and \eqref{BLie1}.
		
		That is, $\Phi_1$ and $\Phi_2$ are well defined and are Lie $K$-homomorphisms.
		
		For the second part, we have that $\im(\Phi_1)=B_N(M)$ and $\im(\Phi_2)=[N,N]$. For this, $\Phi_1$ and $\Phi_2$ are simultaneously surjective if and only if the braided crossed module is perfect.
	\end{proof}

\begin{lem}\label{braidcommumorph}
	Let $(M\xrightarrow{\partial}N, \cdot,\{-,-\})$ be a braided crossed module of Lie $K$-algebras, and $(N\otimes N\xrightarrow{\Id_{N\otimes N}}N\otimes N,[-,-],[-,-])$ the braided crossed module seen in Example~\ref{ExBra}.
		
	Then the pair $(\Phi_1,\Phi_2)$, with $\Phi_1$ and $\Phi_2$ built in Lemma~\ref{braidmorph} is a morphism between $(N\otimes N\xrightarrow{\Id_{N\otimes N}}N\otimes N,[-,-],[-,-])$ and $(M\xrightarrow{\partial}N,\cdot,\{-,-\})$.
	
	Besides $\ker(\Phi_1)\subset (N\otimes N)^{(N\otimes N)}$ and $\ker(\Phi_2)\subset \Z_B(N\otimes N)$.
\end{lem}
	\begin{proof}
		For the proof we will denote the action $[-,-]$ of $N\otimes N\xrightarrow{\Id_{N\otimes N}}N\otimes N$ as $*$ and its braiding as $\lBrack-,-\rBrack$.
		
		First we will show \eqref{XLieH1}. Let $n\otimes n',n''\otimes n'''\in N\otimes N$:
		\begin{align*}
		&\Phi_1((n\otimes n')*(n''\otimes n'''))=\Phi_1([n\otimes n',n''\otimes n'''])=\Phi_1([n,n']\otimes [n'',n'''])\\
		&=\{[n,n'],[n'',n''']\}=\{[n,n'],\partial \{n'',n'''\}\}=[n,n']\cdot \{n'',n'''\}\\
		&=\Phi_2(n\otimes n')\cdot\Phi_1(n''\otimes n'''),
		\end{align*}
		where we use \eqref{BLie1} and \eqref{BLie4}.
		
		Now we will show \eqref{XLieH2}.
		\begin{equation*}
		\partial\circ \Phi_1(n\otimes n')=\partial\{n,n'\}=[n,n']=\Phi_2(\Id_{N\otimes N}(n\otimes n')),
		\end{equation*}
		where we use \eqref{BLie2}.
		
		Now, we will prove \eqref{BXLieH3}.
		\begin{align*}
		&\Phi_1(\lBrack n\otimes n', n''\otimes n'''\rBrack)=\Phi_1([n\otimes n',n''\otimes n'''])=\Phi_1([n,n']\otimes [n'',n'''])\\&=\{[n,n'],[n'',n''']\}=\{\Phi_2(n\otimes n'),\Phi_2(n''\otimes n''')\}.
		\end{align*}
		
		For this, it is proven that $(\Phi_1,\Phi_2)$ is a morphism.

		We will now prove that the inclusions hold.
		
		If $n\otimes n'\in \ker(\Phi_1)$ then $\{n,n'\}=0$. Using \eqref{BLie1} we have that $0=\partial\{n,n'\}=[n,n']$.
		
		Since $(N\otimes N)^{(N\otimes N)}=\{x\in N\otimes N\mid (n''\otimes n''')* x=0, \  n''\otimes n'''\in N\otimes N\}$ (it is enough to work on generators), we have
		\begin{align*}
		(n''\otimes n''')*(n\otimes n')=[n''\otimes n''',n\otimes n']=[n'',n''']\otimes [n,n']=[n'',n''']\otimes 0=0.
		\end{align*}
		With the previous equality we have that $n\otimes n'\in (N\otimes N)^{(N\otimes N)}$ and it is proven that $\ker(\Phi_1)\subset (N\otimes N)^{(N\otimes N)}$.
		
		For the last inclusion, we take $n\otimes n'\in \ker(\Phi_2)$, i.e. $[n,n']=0$.
		
		Since it is enough to work in generators, we have that $$\Z_B(N\otimes N)=\{x\in N\otimes N\mid \lBrack x,n''\otimes n'''\rBrack=0=\lBrack n''\otimes n''',x\rBrack, \   n''\otimes n'''\in N\otimes N \}.$$
		
		Then, since
		\begin{align*}
		&\lBrack n''\otimes n''',n\otimes n' \rBrack=[n''\otimes n''',n\otimes n']=[n'',n''']\otimes [n,n']=[n'',n''']\otimes 0=0,\\
		&\lBrack n\otimes n',n''\otimes n'''\rBrack=[n\otimes n',n''\otimes n''']=[n,n']\otimes [n'',n''']=0\otimes [n'',n''']=0,
		\end{align*}
$n\otimes n'\in Z_B(N\otimes N)$, which proves that $\ker(\Phi_2)\subset Z_B(N\otimes N)$.
	\end{proof}

\begin{cor}\label{extensioniffperfect}
	The morphism given in Lemma~\ref{braidcommumorph} is a central extension if and only if $(M\xrightarrow{\partial}N,\cdot,\{-,-\})$ is a perfect braided crossed module of Lie $K$-algebras.
\end{cor}
	\begin{proof}
		It will be a central extension if and only if $(\Phi_1,\Phi_2)$ is an extension, since in Lemma~\ref{braidcommumorph} were proved the two inclusions and they have the restricted operations as braided crossed module.
		
		But $(\Phi_1,\Phi_2)$ is an extension if and only if $\Phi_1$ and $\Phi_2$ are simultaneously surjective, and by Lemma~\ref{braidmorph} that it happen if and only if the braided crossed module $(M\xrightarrow{\partial}N,\cdot, \{-,-\})$ is perfect.
	\end{proof}

\begin{prop}\label{mediadoraUcentralext}
	If $(X_1\xrightarrow{\delta} X_2,*,\lParen-,-\rParen)\xrightarrow{f=(f_1,f_2)}(M\xrightarrow{\partial}N,\cdot,\{-,-\})$ is a central extension, then we have a morphism in the category of braided crossed modules $h^X\colon (N\otimes N\xrightarrow{\Id_{N\otimes N}}N\otimes N,[-,-],[-,-]) \xrightarrow{\ \ }(X_1\xrightarrow{\delta} X_2,*,\lParen-,-\rParen)$ defined by:
	\begin{itemize}
		\item $h^X_1\colon N\otimes N\xrightarrow{}X_1$, $n\otimes \eta\mapsto \lParen \overline{n},\overline{\eta}\rParen$, where $\overline{n},\overline{\eta}\in X_2$ are  elements  such that $f_2(\overline{n})=n$ and $f_2(\overline{\eta})=\eta$.
		\item $h^X_2\colon N\otimes N\xrightarrow{}X_2$, $n\otimes \eta\mapsto [\overline{n},\overline{\eta}]$, where $\overline{n},\overline{\eta}\in X_2$ are elements such that $f_2(\overline{n})=n$ and $f_2(\overline{\eta})=\eta$.
	\end{itemize}

	In addition $f\circ h^X=\Phi$, i.e, it is a morphism between the extensions.
\end{prop}
	\begin{proof}
		Since $f_1$ and $f_2$ are surjective, we can always take a pair of elements to apply the definition.
		We need to prove that $h^X_1$ and $h^X_2$ are well defined.
		
		We will start with $h^X_1$. We will take $\overline{n},\widetilde{n},\overline{\eta},\widetilde{\eta}\in X_2$ such that $f_2(\overline{n})=f_2(\widetilde{n})=n$ and $f_2(\overline{\eta})=f_2(\widetilde{\eta})=\eta$ and prove that $\lParen \overline{n},\overline{\eta}\rParen=\lParen \widetilde{n},\widetilde{\eta}\rParen$.
		
		Since $f_2(\overline{n})=f_2(\widetilde{n})$ we have that $\overline{n}-\widetilde{n}\in \ker(f_2)$. $f=(f_1,f_2)$ is a central extension, that means, in the braided crossed module of Lie $K$-algebras, that $\ker(f_2)\subset \Z_B(X_2)$. By the definition of $\Z_B(X_2)$ we know that $\lParen\overline{n}-\widetilde{n},\overline{\eta}\rParen=0$ and that means $\lParen\overline{n},\overline{\eta}\rParen=\lParen\widetilde{n},\overline{\eta}\rParen$.
		
		Using the same reasoning we have that $\overline{\eta}-\widetilde{\eta}\in \ker(f_2)\subset \Z_B(X_2)$ and, for that reason, $\lParen\widetilde{n},\overline{\eta}-\widetilde{\eta}\rParen=0$. So $\lParen\widetilde{n},\overline{\eta}\rParen=\lParen\widetilde{n},\widetilde{\eta}\rParen$.
		
		With both equalities, we have that $\lParen\overline{n},\overline{\eta}\rParen=\lParen\widetilde{n},\overline{\eta}\rParen=\lParen\widetilde{n},\widetilde{\eta}\rParen$ and $h^X_1$ is independent of the choice.
		
		Since $\Z_B(X_2)\subset \Z(X_2)$ we can change the proof for $h^X_1$ taking the equalities for $[-,-]$ instead of $\lParen-,-\rParen$ which proof that $h^X_2$ is independent of the choice.
		
		We need to prove that $h_1^X$ and $h_2^X$ preserves the relations. Since $f_2$ is a Lie $K$-homomorphism we have that $f_2(\lambda\overline{n}+\mu\overline{\eta})=f_2(\overline{\lambda n+\mu \eta})=\lambda n+\mu \eta$ with $\lambda,\mu\in K$ and $f_2([\overline{n},\overline{\eta}])=f_2(\overline{[n,\eta]})=[n,\eta]$.
		With the last equalities, and since $h_1^X$ and $h_2^X$ are independent of the choice taken for they definition, we can use the same proof as in Lemma~\ref{braidmorph} to proof that $h^X_1$ and $h^X_2$ are well defined, i.e. they preserve the relations, and, since they are defined on generators, they are Lie $K$-homomorphisms.
	
		We need also to prove that $h^X=(h_1^X,h_2^X)$ is a morphism of braided crossed modules of Lie $K$-algebras.
		
		Using now the same proof as the one done in Lemma~\ref{braidcommumorph}, since we can make the changes in the choice inside the braidings and brackets, we prove that $h^X$ is a morphisms of braided crossed modules of Lie $K$-algebras.
		
		To finish, if $n\otimes \eta\in N\otimes N$, then
		\[f_1\circ h^X_1(n\otimes\eta)=f_1(\lParen\overline{n},\overline{\eta}\rParen)
		=\{f_2(\overline{n}),f_2(\overline{\eta})\}
		=\{n,\eta\}=\Phi_1(n\otimes \eta),\]
		and
		\[f_2\circ h^X_2(n \otimes \eta)=f_2([ \overline{n},\overline{\eta}])=[f_2(\overline{n}),f_2(\overline{\eta})]=[n,\eta]=\Phi_2(n\otimes \eta).\]
		
		That proves that $f\circ h^X_1=\Phi$.
	\end{proof}

\begin{lem}\label{perfecttensor}
	If $N$ is a perfect Lie $K$-algebra, i.e. $N=[N,N]$, then we have that $(N\otimes N\xrightarrow{\Id_{N\otimes N}}N\otimes N,[-,-],[-,-])$ is a perfect braided crossed module of Lie $K$-algebras.
	
	In particular, if $(M\xrightarrow{\partial} N,\cdot,\{-,-\})$ is a perfect braided crossed module, then $(N\otimes N\xrightarrow{\Id_{N\otimes N}}N\otimes N,[-,-],[-,-])$ is a perfect braided crossed module of Lie $K$-algebras.
\end{lem}
	\begin{proof}
	Since the braiding in $(N\otimes N\xrightarrow{\Id_{N\otimes N}}N\otimes N,[-,-],[-,-])$ is the bracket, we have that $[N\otimes N,N\otimes N]=B_{N\otimes N}(N\otimes N)$, so is enough to prove that $[N\otimes N,N\otimes N]=N\otimes N$.
	
	Since $N=[N,N]$ is enough to prove that the generators  $[n_1,n_2]\otimes [n_3,n_4]$ are inside $[N\otimes N,N\otimes N]$. Using \eqref{RTLie4} we have that $[n_1,n_2]\otimes [n_3,n_4]=[n_1\otimes n_2,n_3\otimes n_4]\in [N\otimes N,N\otimes N]$ and the first part of the proof is complete.
	
	For the second part, if $(M\xrightarrow{\partial} N,\cdot,\{-,-\})$ is perfect, then $N=[N,N]$ and we conclude using the first part.
	\end{proof}

\begin{prop}\label{siperfhuni}
	Let $(Y_1\xrightarrow{\varrho} Y_2,\star,\lBrack-,-\rBrack)\xrightarrow{\Psi}(M\xrightarrow{\partial}N, \cdot,\{-,-\})$ be a morphism of braided crossed modules of Lie $K$-algebras such that the braided crossed module $(Y_1\xrightarrow{\varrho} Y_2,\star,\lBrack-,-\rBrack)$ is perfect.
	
	If $(X_1\xrightarrow{\rho} X_2,*,\lParen-,-\rParen)\xrightarrow{f}(M\xrightarrow{\partial}N, \cdot,\{-,-\})$ is a central extension and exists $(Y_1\xrightarrow{\varrho} Y_2,\star,\lBrack-,-\rBrack)\xrightarrow{h}(X_1\xrightarrow{\rho} X_2,*,\lParen-,-\rParen)$ such that $\Psi=f \circ h$, then h is  the unique that satisfy that equality.
\end{prop}
	\begin{proof}
		Suppose that exists $g,h\colon (Y_1\xrightarrow{\varrho} Y_2,\star,\lBrack-,-\rBrack)\xrightarrow{}(X_1\xrightarrow{\rho} X_2,*,\lParen-,-\rParen)$ such that $\Psi=f \circ h=f \circ g$, i.e., we have that $\Psi_1=f_1 \circ h_1=f_1 \circ g_1$ and $\Psi_2=f_2 \circ h_2=f_2 \circ g_2$.
		
		If $y\in Y_2$ then $f_2 \circ h_2(y)=f_2 \circ g_2(y)$ and for that $f_2(h_2(y)-g_2(y))=0$, what means $h_2(y)-g_2(y)\in \ker(f_2)$. Then exists $k_y\in \ker(f_2)$ such that $h_2(y)=g_2(y)+k_y$. Since $f$ is a central extension we have that $\ker(f_2)\subset \Z_B(X_2)\subset \Z(X_2)$. If we take $y,z\in Y_2$ then, since $k_y,k_z\in \Z(X_2)$ we have \[[k_y,g_2(z)]=[k_y,k_z]=[g_2(y),k_z]=0.\]
		
		Using this, we have:
		\begin{align*}
		&h_2([y,z])=[h_2(y),h_2(z)]=[g_2(y)+k_z,g_2(z)+k_z]\\
		&=[g_2(y),g_2(z)]+[k_y,g_2(z)]+[k_y,k_z]+[g_2(y),k_z]=[g_2(y),g_2(z)]=g_2([y,z]).
		\end{align*}

		Since $(Y_1\xrightarrow{\varrho} Y_2,\star,\lBrack-,-\rBrack)$ is perfect, $Y_2=[Y_2,Y_2]$, which implies, with the previous equality, that $g_2=h_2$.
		
		In addition, since $\ker(f_2)\subset \Z_B(X_2)$, if $y,z\in Y_2$ we have that:
		\begin{align*}
		&h_1(\lBrack y,z\rBrack)=\lParen h_2(y),h_2(z)\rParen=\lParen g_2(y)+k_z,g_2(y)+k_z\rParen\\
		&=\lParen g_2(y),g_2(z)\rParen+\lParen k_y,g_2(z)\rParen +\lParen k_y,k_z\rParen +\lParen g_2(y),k_z\rParen =\lParen g_2(y),g_2(z)\rParen=g_1(\lBrack y,z\rBrack),
		\end{align*}
		where we use that $k_y,k_z\in \Z_B(X_2)$ in $\lParen k_y,g_2(z)\rParen=\lParen k_y,k_z\rParen=\lParen g_2(y),k_z\rParen=0$.
		
		Since $(Y_1\xrightarrow{\varrho} Y_2,\star,\lBrack-,-\rBrack)$ is perfect, $Y_1=B_{Y_2}(Y_1)$, which implies that it is generated by the images of the braiding. With the previous equality, that proves that $g_1=h_1$.
	\end{proof}

\begin{cor}\label{ifperhaveuniv}
	If $(M\xrightarrow{\partial}N,\cdot,\{-,-\})$ is a perfect braided crossed module of Lie $K$-algebras, then $$(N\otimes N\xrightarrow{\Id_{N\otimes N}}N\otimes N,[-,-],[-,-])\xrightarrow{\Phi}(M\xrightarrow{\partial}N,\cdot,\{-,-\})$$ is the universal central extension, where $\Phi_1,\Phi_2$ were defined in Lemma~\ref{braidmorph}.
\end{cor}
	\begin{proof}
		Since $(M\xrightarrow{\partial}N,\cdot,\{-,-\})$ is perfect, Corollary~\ref{extensioniffperfect} says that the morphism $(N\otimes N\xrightarrow{\Id_{N\otimes N}}N\otimes N,[-,-],[-,-])\xrightarrow{\Phi}(M\xrightarrow{\partial}N,\cdot,\{-,-\})$ is an extension.
		
		We need to prove that it is universal.
		
		Starting with existence of the morphism, if we have another central extension $(X_1\xrightarrow{\delta}X_2,*,\lParen-,-\rParen)\xrightarrow{f}(M\xrightarrow{\partial}N,\cdot,\{-,-\})$ then using Proposition~\ref{mediadoraUcentralext} we have $h^X$ such that $\Phi=f\circ h^X$.
		
		The uniqueness of this morphism is given by Proposition~\ref{siperfhuni}. We can use the previous proposition since $(N\otimes N\xrightarrow{\Id_{N\otimes N}}N\otimes N,[-,-],[-,-])$ is perfect by the use of Lemma~\ref{perfecttensor} and the fact that $(M\xrightarrow{\partial}N,\cdot,\{-,-\})$ is perfect.
	\end{proof}

In Corollary~\ref{ifperhaveuniv} we prove that if $(M\xrightarrow{\partial}N,\cdot,\{-,-\})$ is perfect, then it has an universal central extension. Lets see now the other implication.

\begin{prop}\label{experentoper}
	Let $(Y_1\xrightarrow{\varrho} Y_2,\star,\lBrack-,-\rBrack)\xrightarrow{\ \Psi \ }(M\xrightarrow{\partial}N, \cdot,\{-,-\})$ be an extension of braided crossed modules of Lie $K$-algebras such that the braided crossed module $(Y_1\xrightarrow{\varrho} Y_2,\star,\lBrack-,-\rBrack)$ is perfect.
	
	Then $(M\xrightarrow{\partial}N,\cdot,\{-,-\})$ is perfect.
\end{prop}
	\begin{proof}
		Since $\Psi$ is an extension, $\Psi_1$ and $\Psi_2$ are surjective maps. Using this fact, we know that $\im(\Psi_1)=M$ and $\im(\Psi_2)=N$.
		
		$(Y_1\xrightarrow{\varrho} Y_2,\star,\lBrack-,-\rBrack)$ is perfect, that means, $Y_1=B_{Y_2}(Y_1)$ and $Y_2=[Y_2,Y_2]$.
		
		Since the elements $\lBrack y,z\rBrack$, with $y,z\in Y_2$ are the generators of $Y_1$, we know that $\Psi(\lBrack y,z\rBrack)$ are the generators of $\im(\Psi_1)=M$. Since $\Psi_1(\lBrack y,z\rBrack)=\{\Phi_2(y),\Phi_2(z)\}$, which means that the generators of $M$ are braid elements and $M\subset B_N(M)$. With this we have $M=B_N(M)$.
		
		For $N$ we know that elements $[y,z]$, with $y,z\in Y_2$, are the generators of $Y_2$. Using this, we know that $\Phi_2([y,z])=[\Phi_2(y),\Phi_2(z)]$ are the generators of $\im(\Phi_2)=N$, which means that $N\subset[N,N]$, and then $N=[N,N]$.
		
		Then we have that $(M\xrightarrow{\partial}N\cdot,\{-,-\})$ is perfect.
	\end{proof}

\begin{lem}\label{ifnotpeefthen2}
	Let $(Y_1\xrightarrow{\varrho} Y_2,\star,\lBrack-,-\rBrack)\xrightarrow{ \ \Psi \ }(M\xrightarrow{\partial}N, \cdot,\{-,-\})$ be a central extension of braided crossed modules of Lie $K$-algebras such that the braided crossed module $(Y_1\xrightarrow{\varrho} Y_2,\star,\lBrack-,-\rBrack)$ is not perfect.
	
	Then exists another extension $(X_1\xrightarrow{\delta} X_2,*,\lParen-,-\rParen)\xrightarrow{\ f \ }(M\xrightarrow{\partial}N, \cdot,\{-,-\})$ and two different morphisms $h,g\colon(Y_1\xrightarrow{\varrho} Y_2,\star,\lBrack-,-\rBrack)\xrightarrow{}(X_1\xrightarrow{\delta} X_2,*,\lParen-,-\rParen)$ such that $\Psi=f\circ h=f\circ g$.
\end{lem}
	\begin{proof}
		Let $(B_{Y_2}(Y_1)\xrightarrow{\varrho|_{B_{Y_2}(Y_1)}}[Y_2,Y_2],\star_C,\lBrack-,-\rBrack_C)\xrightarrow{i=(i_1,i_2)}(Y_1\xrightarrow{\varrho}Y_2,\star,\lBrack-,-\rBrack)$ be the inclusion morphism of the commutator braided crossed submodule.
		
		If we take the cokernel of $i$ we have the crossed module of Lie $K$-algebras $(\frac{Y_1}{B_{Y_2}(Y_1)}\xrightarrow{\overline{\varrho}}\frac{Y_2}{[Y_2,Y_2]},\overline{\star},\lBrack-,-\rBrack)$. We will abuse of notation and we will denote in the same way the braiding and the braiding in the quotient. We will denote the elements in $\frac{Y_1}{B_{Y_2}(Y_1)}$ as $ \overline{x}$, $x\in Y_1$, and the ones in $\frac{Y_2}{[Y_2,Y_2]}$ as $\overline{y}$, $y\in Y_2$.
		
		We take now the product in the category of braided crossed modules and we build $(M\xrightarrow{\partial}N,\cdot,\{-,-\})\times (\frac{Y_1}{B_{Y_2}(Y_1)}\xrightarrow{\overline{\varrho}}\frac{Y_2}{[Y_2,Y_2]},\overline{\star},\lBrack-,-\rBrack)$. We denote as $\pi^1$ the first projection morphism.
		 Since $\pi^1_1$ and $\pi_2^1$ are surjective maps, we have that $$(M\xrightarrow{\partial}N,\cdot,\{-,-\})\times (\frac{Y_1}{B_{Y_2}(Y_1)}\xrightarrow{\overline{\varrho}}\frac{Y_2}{[Y_2,Y_2]},\overline{\star},\lBrack-,-\rBrack)\xrightarrow{\pi^1}(M\xrightarrow{\partial}N,\cdot,\{-,-\})$$ is an extension. We will denote the braiding in the product as $\lBrace-,-\rBrace$.
		
		We will prove that it is a central extension, i.e., we need to prove the inclusions $\ker(\pi^1_1)\subset (M\times \frac{Y_1}{B_{Y_2}(Y_2)})^{(N\times \frac{Y_2}{[Y_2,Y_2]})}$ and $\ker(\pi_2^1)\subset \Z_B(N\times \frac{Y_2}{[Y_2,Y_2]})$.
		
		If $a\in\ker(\pi_1^1)$ then $a=(0,\overline{x})$ with $x\in Y_1$.
		
		If we take $(n,\overline{y})\in N\times \frac{Y_2}{[Y_2,Y_2]}$ then:
		\begin{equation*}
		(n,\overline{y})(\cdot\times\overline{\star})(0,\overline{x})=(n\cdot 0,\overline{y}\overline{\star}\overline{x})=(0,\overline{y\star x}).
		\end{equation*}
		But $\overline{y\star x}=\overline{0}$ since $y\star x\in D_{Y_2}(Y_1)\subset B_{Y_2}(Y_1)$.

\,
		
		So $\ker(\pi^1_1)\subset \big(M\times \frac{Y_1}{B_{Y_2}(Y_2)}\big)^{\big(N\times \frac{Y_2}{[Y_2,Y_2]}\big)}$.
		
		If $a\in\ker(\pi_2^1)$ then $a=(0,\overline{y})$ with $y\in Y_2$. If we take $(n,\overline{y_1})\in N\times \frac{Y_2}{[Y_2,Y_2]}$ then:
		\begin{align*}
		&\lBrace(0,\overline{y}),(n,\overline{y_1})\rBrace=(\{0,n\},\lBrack\overline{y},\overline{y_1}\rBrack)=(0,\overline{\lBrack y,y_1\rBrack}),\\
		&\lBrace(n,\overline{y_1}),(0,\overline{y})\rBrace=(\{n,0\},\lBrack\overline{y_1},\overline{y}\rBrack)=(0,\overline{\lBrack y_1,y\rBrack}).
		\end{align*}
		But $\overline{\lBrack y,y_1\rBrack}=\overline{\lBrack y_1,y\rBrack}=\overline{0}$ since $\lBrack y_1,y\rBrack,\lBrack y,y_1\rBrack\in B_{Y_2}(Y_1)$. Therefore $\ker(\pi_2^1)\subset \Z_B(N\times \frac{Y_2}{[Y_2,Y_2]})$, and $\pi^1$ is a central extension.
		
		If we denote as $i^c\colon (Y_1\xrightarrow{\varrho}Y_2,\star,\lBrack-,-\rBrack)\xrightarrow{}(\frac{Y_1}{B_{Y_2}(Y_1)}\xrightarrow{\overline{\varrho}}\frac{Y_2}{[Y_2,Y_2]},\overline{\star},\lBrack-,-\rBrack)$ the cokernel map of $i$
		we have two morphisms, induced by the product, with domain $(Y_1\xrightarrow{\rho} Y_2,\star,\lBrack-,-\rBrack)$ and  $(M\xrightarrow{\partial}N,\cdot,\{-,-\})\times (\frac{Y_1}{B_{Y_2}(Y_1)}\xrightarrow{\overline{\rho}}\frac{Y_2}{[Y_2,Y_2]},\overline{\star},\lBrack-,-\rBrack)$ as codomain. They are $h=(\Psi, 0)$ and $g=(\Psi,i^c)$. Since they are the induced by the universal property of the product, is immediate that $\Psi=\pi^1\circ h=\pi^1\circ g$.
		
		To finish the proof we only have to prove that they are different. Since the braided crossed module $(Y_1\xrightarrow{\rho}Y_2,\star,\lBrack-,-\rBrack)$ is not perfect and $i^c_1$ and $i^c_2$ are surjective we know that $i^c_1\neq 0$ or $i^c_2\neq 0$ (if both were the zero morphism then $(Y_1\xrightarrow{\rho}Y_2,\star,\lBrack-,-\rBrack)$ would be perfect) and, for that $h\neq g$.
	\end{proof}

\begin{cor}\label{ifnotperthennotuniv}
	If  $(M\xrightarrow{\partial}N,\cdot,\{-,-\})$ is a braided crossed module, then its universal central extension, if exists, is perfect.
\end{cor}
	\begin{proof}
		If the universal extension is not perfect, then using Lemma~\ref{ifnotpeefthen2} we have another central extension such that we have two different morphism from the universal central extension satisfying  the commutative diagram of its definition, which enter in contradiction with the universality.
	\end{proof}

\begin{thm}
	A braided crossed module of Lie $K$-algebras has universal central extension if and only if it is perfect.
\end{thm}
	\begin{proof}
		If the braided crossed module is perfect then using Corollary~\ref{ifperhaveuniv} we can build its universal central extension.
		
		If the braided crossed module has universal central extension, then using Corollary~\ref{ifnotperthennotuniv} we have that the universal central extension is perfect. Since it is a extension, we can use Lemma~\ref{experentoper} and conclude that the braided crossed module is perfect.
	\end{proof}

\section{Braiding on universal extension of crossed modules of Lie $K$-algebras}

In \cite{Fuk} did not talk about a universal central extension of braided crossed modules of groups. He takes the result of Norrie (\cite{Norr}) about crossed modules of groups an he built a canonical braiding on the universal central extension of crossed modules of groups when the crossed module of which he make the extension is braided too and proved that it was universal in a sense we will explain in this section.

In this part of the paper we will do the same, we will consider braided crossed modules extensions, but as he, we will try to built a braiding on the universal central extension of a braided crossed module though as crossed module, which we will call compatible central extensions. In this sense we will replicate part of the results given by Fukushi in \cite{Fuk}.

Casas and Ladra in \cite{PCMLie} prove that the universal central extension of a perfect crossed module of Lie $K$-algebras $(M\xrightarrow{}N,\cdot)$ is given by:
\begin{equation*}
(N\otimes M\xrightarrow{\Id_N\otimes \partial}N\otimes N,*)\xrightarrow{c=(c_1,c_2)}(M\xrightarrow{\partial}N,\cdot),
\end{equation*}
where $N\otimes M$ is given by the actions $\cdot$ of $N$ on $M$ and $m\star n=[\partial(m),n]$ of $M$ on $N$; the action is given by $(n\otimes n')*(n''\otimes m)=[[n,n'],n'']\otimes m+n''\otimes [n,n']\cdot m$ for $n,n',n''\in N, m\in M$; and the morphisms are $c_1(n\otimes m)=n\cdot m$ and $c_2(n\otimes n')=[n,n']$.

With this in mind we have:
\begin{prop}\label{braidunivxmod}
	If $(M\xrightarrow{\partial \ }N,\cdot, \{-,-\})$ is a braided crossed module of Lie $K$-algebras then $\lBrace -,- \rBrace \colon(N\otimes N)\times (N\otimes N)\xrightarrow{}N\otimes M$ defined over generators as $\lBrace n\otimes n',n''\otimes n'''\rBrace=[n,n']\otimes\{n'',n'''\}$ is a braiding for the crossed module of Lie $K$-algebras $(N\otimes M\xrightarrow{\Id_N\otimes \partial}N\otimes N,*)$.
\end{prop}
	\begin{proof}
		We need to prove that $\lBrace-,-\rBrace$ is well defined. For that we need to prove that it preserves the relations. Is immediate that it preserves \eqref{RTLie1} and \eqref{RTLie2} using the $K$-bilinearity of $[-,-]$ and $\{-,-\}$. \eqref{RTLie3} and \eqref{RTLie4} are verified too since $\{-,-\}$ and $[-,-]$ verifies it.
		
		We need then to prove the axioms of braidings.
		
		Let $n,n,n',n''\in N$, $m,m'\in M$. Then:
		\begin{align*}
		&(\Id_N\otimes \partial)(\lBrace n\otimes n',n''\otimes n'''\rBrace)=(\Id_N\otimes \partial)([n,n']\otimes \{n'',n'''\})\\
		&=[n,n']\otimes\partial\{n'',n'''\}=[n,n']\otimes[n'',n''']=[n\otimes n',n''\otimes n'''],
		\end{align*}
		\begin{align*}
		&\lBrace(\Id_N\otimes\partial)(n\otimes m),(\Id_N\otimes\partial)(n'\otimes m')\rBrace=\lBrace n\otimes \partial(m),n'\otimes\partial(m')\rBrace\\
		&=[n,\partial(m)]\otimes\{n',\partial(m')\}=-(m\star n)\otimes (n'\cdot m')=[n\otimes m,n'\otimes m'],
		\end{align*}
		\begin{align*}
		&\lBrace(\Id_N\otimes\partial)(n\otimes m),n'\otimes n''\rBrace=\lBrace n\otimes\partial(m),n'\otimes n''\rBrace=[n,\partial(m)]\otimes\{n',n''\}\\
		&=-(m\star n)\otimes\{n',n''\}=n\otimes [m,\{n',n''\}]-\{n',n''\}\star n\otimes m\\
		&=n\otimes \{\partial(m),\partial(\{n',n''\})\}-[\partial(\{n',n''\}),n]\otimes m\\
		&=-n\otimes [n',n'']\cdot m-[[n',n''],n]\otimes m=-(n'\otimes n'')*(n\otimes m),
		\end{align*}
		where we use in the third equality the second relation of \eqref{RTLie3}.
		\begin{align*}
		&\lBrace n'\otimes n'',(\Id_N\otimes\partial)(n\otimes m)\rBrace=\lBrace n'\otimes n'',n\otimes \partial(m)\rBrace=[n',n'']\otimes \{n,\partial(m)\}\\
		&=[n',n'']\otimes (n\cdot m)=n\otimes [n,n']\cdot m+[[n',n''],n]\otimes m\\
		&=(n'\otimes n'')*(n\otimes m),
		\end{align*}
		where we use in the third equality the first relation of \eqref{RTLie3}.
		\begin{align*}
		&\lBrace n_1\otimes n'_1,[n_2\otimes n'_2,n_3\otimes n'_3]\rBrace=\lBrace n_1\otimes n'_1,[n_2\otimes n'_2]\otimes[n_3\otimes n'_3]\rBrace\\
		&=[n_1,n'_1]\otimes \{[n_2,n'_2],[n_3,n'_3]\}=[n_1,n'_1]\otimes [\{n_2,n'_2\},\{n_3,n'_3\}]\\
		&=(\{n_3,n'_3\}\star [n_1,n'_1])\otimes\{n_2,n'_2\}-(\{n_2,n'_2\}\star [n_1,n'_1])\otimes \{n_3,n'_3\}\\
		&=[\partial(\{n_3,n'_3\}),[n_1,n'_1]]\otimes\{n_2,n'_2\}-[\partial(\{n_2,n'_2\}),[n_1,n'_1]]\otimes\{n_3,n'_3\}\\
		&=[[n_3,n'_3],[n_1,n'_1]]\otimes\{n_2,n'_2\}-[[n_2,n'_2],[n_1,n'_1]]\otimes\{n_3,n'_3\}\\
		&=-[[n_1,n'_1],[n_3,n'_3]]\otimes\{n_2,n'_2\}+[[n_1,n'_1],[n_2,n'_2]]\otimes\{n_3,n'_3\}\\
		&=-\lBrace[n_1,n'_1]\otimes  [n_3,n'_3],n_2\otimes n'_2\rBrace+\lBrace[n_1,n'_1]\otimes  [n_2,n'_2],n_3\otimes n'_3\rBrace\\
		&=\lBrace[n_1\otimes n'_1,n_2\otimes n'_2],n_3\otimes n'_3\rBrace-\lBrace[n_1\otimes n'_1,n_3\otimes n'_3],n_2\otimes n'_2\rBrace,
		\end{align*}
		\begin{align*}
		&\lBrace[n_1\otimes n'_1,n_2\otimes n'_2],n_3\otimes n'_3\rBrace=\lBrace[n_1,n'_1]\otimes [n_2, n'_2],n_3\otimes n'_3\rBrace\\
		&=[[n_1,n'_1],[n_2,n'_2]]\otimes \{n_3,n'_3\}\\
		&=[n_1,n'_1]\otimes [n_2,n'_2]\cdot \{n_3,n'_3\}-[n_2,n'_2]\otimes [n_1,n'_1]\cdot \{n_3,n'_3\}\\
		&=[n_1,n'_1]\otimes \{[n_2,n'_2],\partial(\{n_3,n'_3\})\}-[n_2,n'_2]\otimes \{[n_1,n'_1],\partial(\{n_3,n'_3\})\}\\
		&=[n_1,n'_1]\otimes \{[n_2,n'_2],[n_3,n'_3]\}-[n_2,n'_2]\otimes \{[n_1,n'_1],[n_3,n'_3]\}\\
		&=\lBrace n_1\otimes n'_1,[n_2,n'_2]\otimes[n_3,n'_3]\rBrace-\lBrace n_2\otimes n'_2,[n_1,n'_1]\otimes[n_3,n'_3]\rBrace\\
		&=\lBrace n_1\otimes n'_1,[n_2\otimes n'_2,n_3\otimes n'_3]\rBrace-\lBrace n_2\otimes n'_2,[n_1\otimes n'_1, n_3\otimes n'_3]\rBrace.
		\end{align*}
		In all equalities we use the properties of $\{-,-\}$ and relations of the tensor product.
		
		With this equalities we conclude that $\lBrace -,-\rBrace$ is a braiding.
	\end{proof}

\begin{prop}\label{Compunivex}
	If $(M\xrightarrow{\partial}N,\cdot,\{-,-\})$ is a braided crossed module of Lie $K$-algebras such that the crossed module $(M\xrightarrow{}N,\cdot)$ is a perfect crossed module, then
	$(N\otimes M\xrightarrow{\Id_N\otimes \partial}M\otimes M,*,\lBrace -,-\rBrace)\xrightarrow{c=(c_1,c_2)}(M\xrightarrow{\partial}N,\cdot,\{-,-\})$ is a compatible central extension, where $c_1(n\otimes m)=n\cdot m$ and $c_2(n\otimes n')=[n,n']$ and $\lBrace -,-\rBrace$ is defined in Proposition~\ref{braidunivxmod}.
\end{prop}
	\begin{proof}
		Since $(M\xrightarrow{}N,\cdot)$ is a perfect crossed module we have the central extension of crossed modules $(N\otimes M\xrightarrow{\Id_N\otimes \partial}M\otimes M,*)\xrightarrow{c=(c_1,c_2)}(M\xrightarrow{\partial}N,\cdot)$ (\cite{PCMLie}).
		
		For that, if we prove that $c$ respects the braiding it will be prove that it is a compatible central extension.
		\begin{align*}
		&c_1(\lBrace n_1\otimes n'_1,n_2\otimes n'_2\rBrace)=c_1([n_1,n'_1]\otimes\{n_2,n'_2\})=[n_1,n'_1]\cdot \{n_2,n'_2\}\\&=\{[n_1,n'_1],\partial(\{n_2,n'_2\})\}=\{[n_1,n'_1],[n_2,n'_2]\}=\{c_2(n_1\otimes n_1),c_2(n_2\otimes n'_2)\}.
		\end{align*}
		Then $c$ is a compatible central extension.
	\end{proof}

Now we will give the analogous result to the one Fukushi did in \cite{Fuk} for the case of central extensions in braided crossed modules of groups.

\begin{prop}\label{univcompcentrext}
	If $(M\xrightarrow{\partial}N,\cdot,\{-,-\})$ is a braided crossed module of Lie $K$-algebras such that $(M\xrightarrow{\partial}N,\cdot)$ is a perfect crossed module of Lie $K$-algebras,
	 then $(N\otimes M\xrightarrow{\Id_N\otimes \partial}N\otimes N,*,\lBrace -,-\rBrace )\xrightarrow{c=(c_1,c_2)}(M\xrightarrow{\partial}N,\cdot,\{-,-\})$, defined in Proposition~\ref{Compunivex} is the universal compatible central extension, in the sense that if
	 $(X_1\xrightarrow{\delta}X_2,\diamond,\lParen-,-\rParen)\xrightarrow{f}(M\xrightarrow{\partial}N,\cdot,\{-,-\})$
is another compatible central extension, then exist an unique homomorphism of braided crossed modules of Lie $K$-algebras
 $(N\otimes M\xrightarrow{\Id_N\otimes \partial}N\otimes N,*,\lBrace -,-\rBrace )\xrightarrow{h=(h_1,h_2)}(X_1\xrightarrow{\delta}X_2,\diamond,\lParen-,-\rParen)$ such that $c=f\circ h$.
	
	In addition, $h$ is the same morphism as in the universality of the non-braiding case.
\end{prop}
	\begin{proof}
		Let $(X_1\xrightarrow{\delta}X_2,\diamond,\lParen-,-\rParen)\xrightarrow{f}(M\xrightarrow{\partial}N,\cdot,\{-,-\})$ be a compatible central extension of Lie $K$-algebras.
		
		We have, by definition, that $(X_1\xrightarrow{\delta}X_2,\diamond)\xrightarrow{f}(M\xrightarrow{\partial}N,\cdot)$ is a central extension of crossed modules.
		
		If $(M\xrightarrow{\partial}M,\cdot)$ is perfect, with the results in \cite{PCMLie} we know that the morphism $(N\otimes M\xrightarrow{\Id_N\otimes \partial}N\otimes N,*)\xrightarrow{c=(c_1,c_2)}(M\xrightarrow{\partial}N,\cdot)$ is the universal central extension of crossed modules.
		
		We know, using again the results of Casas and Ladra (\cite{PCMLie}), that exists an unique morphism of crossed modules of Lie $K$-algebras
		$$(N\otimes M\xrightarrow{\Id_N\otimes \partial}N\otimes N,*)\xrightarrow{h=(h_1,h_2)}(X_1\xrightarrow{\delta}X_2,\diamond),$$
		defined as $h_1(n\otimes m)=\overline{n}\diamond\overline{m}$ and $h_2(n\otimes n')=[\overline{n},\overline{n'}]$ where $f_1(\overline{m})=m$, $f_2(\overline{n})=n$ and $f_2(\overline{n'})=n'$; which verify $c=f\circ h$.
		
		If $h$ is a braided crossed module homomorphism with the braidings $\lBrace -,-\rBrace$ and $\lParen-,-\rParen$, then we have the existence of morphism $h$.
		
		Let $n_1,n_2,\eta_1,\eta_2\in N$. Then:
		\begin{equation*}
		h_1(\lBrack n_1\otimes \eta_1,n_2\otimes \eta_2\rBrack)=h_1([n_1,\eta_1]\otimes \{n_2,\eta_2\})=\overline{[n_1,\eta_1]}\diamond \overline{\{n_2,\eta_2\}}.
		\end{equation*}
		Since $h_1$ and $h_2$ are well defined, and we have that $f_2(\overline{[n_1,\eta_1]})=[n_1,\eta_1]=[f_2(\overline{n_1}),f_2(\overline{\eta_1})]=f_2([\overline{n_1},\overline{\eta_1}])$ for $f_2$ be Lie $K$-homomorphisms and $f_1(\overline{\{n_2,\eta_2\}})=\{n_2,\eta_2\}=\{f_2(\overline{n_1}),f_2(\overline{\eta_1})\}=f_1(\lParen\overline{n_1},\overline{\eta_1}\rParen)$ for be $f$ a morphism of braided crossed modules; we can continue the previous equalities:
		\begin{align*}
		&h_1(\lBrace n_1\otimes \eta_1,n_2\otimes \eta_2\rBrace)=\overline{[n_1,\eta_1]}\diamond \overline{\{n_2,\eta_2\}}=[\overline{n_1},\overline{\eta_1}]\diamond\lParen \overline{n_2},\overline{\eta_2}\rParen\\
		&=\{[\overline{n_1},\overline{\eta_1}],\delta(\lParen\overline{n_2},\overline{\eta_2}\rParen)\}=\{[\overline{n_1},\overline{\eta_1}],[\overline{n_2},\overline{\eta_2}]\}=\{h_2(n_1\otimes \eta_1),h_2(n_2\otimes \eta_2)\}.
		\end{align*}
		
		Then the same pair of homomorphisms as in the non-braiding case, verifies in the braiding case that $c=f\circ h$.
		
		We need to prove the uniqueness but, since in crossed modules is unique, it must be unique in braided crossed modules, because if it exists another in braiding crossed modules, then forgetting the braidings we would contradict the universality of $c$ as a central extension of crossed modules.
	\end{proof}

Now we will prove that the universal compatible central extension exists if and only if the braided crossed module is perfect as a crossed module.

\begin{prop}\label{experentopercomp}
	Let $(Y_1\xrightarrow{\rho} Y_2,\star,\lBrack-,-\rBrack)\xrightarrow{\Psi}(M\xrightarrow{\partial}N, \cdot,\{-,-\})$ be a extension of braided crossed modules of Lie $K$-algebras such that $(Y_1\xrightarrow{\rho} Y_2,\star)$ is perfect as crossed module.
	Then $(M\xrightarrow{\partial}N,\cdot)$ is perfect.
\end{prop}
	\begin{proof}
		Since $\Psi$ is an extension, $\Psi_1$ and $\Psi_2$ are surjective maps. Using this fact, we know that $\im(\Psi_1)=M$ and $\im(\Psi_2)=N$.
		
		$(Y_1\xrightarrow{\rho} Y_2,\star)$ is perfect, that means, $Y_1=D_{Y_2}(Y_1)$ and $Y_2=[Y_2,Y_2]$.
		
		Since the elements $y\star z$, with $y\in Y_2$ and $z\in Y_1$ are the generators of $Y_1$, we know that $\Psi_1(y\star z)$ are the generators of $\im(\Psi_1)=M$. $\Psi_1(y\star z)=\Phi_2(y)\cdot\Phi_1(z)$, which means that the generators of $M$ are elements of the type $n\cdot m$, which implies that $M\subset D_N(M)$ and then $M=D_N(M)$.
		
		For $N$ we know that elements $[y,z]$, with $y,z\in Y_2$, are the generators of $Y_2$. Using this, we know that $\Phi_2([y,z])=[\Phi_2(y),\Phi_2(z)]$ are the generators of $\im(\Phi_2)=N$, which means that $N\subset[N,N]$, and then $N=[N,N]$.
		
		Then we have that $(M\xrightarrow{\partial}N,\cdot)$ is perfect.
	\end{proof}

\begin{lem}\label{ifnotpeefthen2comp}
	Let $(Y_1\xrightarrow{\rho} Y_2,\star,\lBrack-,-\rBrack)\xrightarrow{\Psi}(M\xrightarrow{\partial}N, \cdot,\{-,-\})$ be a compatible central extension of braided crossed modules of Lie $K$-algebras such that the crossed module $(Y_1\xrightarrow{\rho} Y_2,\star)$ is not perfect.
	
Then exists another compatible central extension \[(X_1\xrightarrow{\delta} X_2,*,\lParen-,-\rParen)\xrightarrow{f}(M\xrightarrow{\partial}N, \cdot,\{-,-\})\] and two different morphisms \[h,g\colon(Y_1\xrightarrow{\rho} Y_2,\star,\lBrack-,-\rBrack)\xrightarrow{}(X_1\xrightarrow{\delta} X_2,*,\lParen-,-\rParen)\] such that $\Psi=f\circ h=f\circ g$.
\end{lem}
	\begin{proof}
		If $(Y_1\xrightarrow{\rho} Y_2,\star,\lBrack-,-\rBrack)$ is a braided crossed module, then we know that the crossed module $(D_{Y_2}(Y_1)\xrightarrow{\rho|_{D_{Y_2}(Y_1)}}[Y_2,Y_2],\star_C)$ is a crossed submodule of $(Y_1\xrightarrow{\rho} Y_2,\star)$. But it is itself a braided crossed submodule of $(Y_1\xrightarrow{\rho} Y_2,\star,\lBrack-,-\rBrack)$ since, if we have $[y,y'],[z,z']\in [Y_2,Y_2]$, then:
		\begin{equation*}
		\lBrack [y,y'],[z,z']\rBrack=\lBrack [y,y'], \rho( \lBrack z,z'\rBrack)\rBrack=[y,y']\star \lBrack z,z'\rBrack \in D_{Y_2}(Y_1).
		\end{equation*}
		Lets denote $i\colon( D_{Y_2}(Y_1)\xrightarrow{\rho|_{D_{Y_2}(Y_1)}}[Y_2,Y_2],\star_C,\lBrack-,-\rBrack_C)\xrightarrow{}(Y_1\xrightarrow{\rho} Y_2,\star,\lBrack-,-\rBrack)$.
		
		Then we can take as a new crossed module the product of the $\coker(i)$ with $(M\xrightarrow{\partial}N,\cdot,\{-,-\})$ whit the extension given by the second projection. We can use the result given in \cite{PCMLie} to say that it will be a compatible central extension, since in that cite they prove that it was a central extension of braided crossed modules.
		
		If $(Y_1\xrightarrow{\rho} Y_2,\star)$ is not perfect then in \cite{PCMLie} were given two maps, they are the induced in the product by the projection in the quotient and $\Psi$ and the zero morphisms and $\Psi$. Since the category of the braided crossed had the same product with induced braiding we have  that morphisms are, in fact, braided crossed morphisms. They are different and both makes $\Psi=f\circ h=f\circ g$.
	\end{proof}

\begin{cor}\label{ifnotperthennotunivcomp}
	The universal compatible central extension of the braided crossed module $(M\xrightarrow{\partial}N,\cdot,\{-,-\})$, if exists, is perfect as a crossed module.
\end{cor}
	\begin{proof}
		If the universal extension is not perfect, then using Lemma~\ref{ifnotpeefthen2comp} we have another compatible central extension such that we have two different morphism from the universal compatible central extension which verifies the commutative diagram of its definition, which enter in contradiction with the universality.
	\end{proof}

With this we can end this section:

\begin{cor}
	A braided crossed module of Lie $K$-algebras has universal compatible central extension if and only if it is perfect as crossed module.
\end{cor}
	\begin{proof}
		If the braided crossed module is perfect then using Proposition~\ref{univcompcentrext} we have its universal central extension.
		
		If the braided crossed module have universal central extension, then using Corollary~\ref{ifnotperthennotunivcomp} we have that the universal compatible central extension is perfect as crossed module. Since it is a extension, we can use Lemma~\ref{experentopercomp} and conclude that our braided crossed module is perfect as a crossed module.
	\end{proof}

\section{Relationship between universal central extension and universal compatible central extension in the braided case}

In this section we will show the relation between the notions of universal central extension and universal compatible central extension in the case of Lie $K$-algebras.

\begin{lem}\label{perfissame}
	Let $(M\xrightarrow{\partial}N,\cdot, \{-,-\})$ be a braided crossed module of Lie $K$-algebras. Then,
	$(M\xrightarrow{\partial}N,\cdot, \{-,-\})$ is perfect if and only if $(M\xrightarrow{\partial}N,\cdot)$ is perfect.
	
\noindent In fact, we have that if $N=[N,N]$ then $B_N(M)=D_N(M)$.
\end{lem}
	\begin{proof}
		We have $N=[N,N]$ because in both of cases they are perfect. For that we only need to check the $M$.
		
		If $(M\xrightarrow{\partial}N,\cdot)$ is perfect, then $D_N(M)=M$. Since we know that $D_N(M)\subset B_N(M)\subset M$, we have that $B_N(M)=M$ and $(M\xrightarrow{\partial}N,\cdot, \{-,-\})$ is perfect.
		
		In the other hand, since $(M\xrightarrow{\partial}N,\cdot, \{-,-\})$ is perfect we have that $B_N(M)=M$. For that we only need to prove that with this hypothesis $D_N(M)=B_N(M)$. We know that $D_N(M)\subset B_N(M)$ so we need to prove that $B_N(M) \subset D_N(M)$.
		
		Then if $\{n,n'\}$ is a generator of $B_N(M)$, since $N=[N,N]$ by be perfect, we have that $\{n,n'\}=\{[n_1,n_2],[n'_1,n'_2]\}$ and:
		\begin{equation*}
		\{[n_1,n_2],[n'_1,n'_2]\}=\{[n_1,n_2],\partial\{n'_1,n'_2\}\}=[n_1,n_2]\cdot \{n'_1,n'_2\}\in D_N(M).
		\end{equation*}
		So, $B_N(M)\subset D_N(M)$ and $(M\xrightarrow{\partial}N,\cdot)$ is perfect.
	\end{proof}

\begin{lem}\label{perfissamecenter}
	Let $(X_1\xrightarrow{\delta}X_2,\diamond,\lParen-,-\rParen)\xrightarrow{f}(M\xrightarrow{\partial}N,\cdot, \{-,-\})$ be an extension of braided crossed modules of Lie $K$-algebras. Then, if $(M\xrightarrow{\partial}N,\cdot,\{-,-\})$ is perfect we have that
	$f$ is a central extension if and only if $f$ is a compatible central extension.
	
\noindent In fact, if $N=[N,N]$ then $\Z_B(N)=Z(N)\cap \st_N(M)$.
\end{lem}
	\begin{proof}
		If $f$ is a central extension, then $\ker(f_1)\subset M^N$ and $\ker(f_2)\subset \Z_B(N)$. In general $\Z_B(N)\subset \st_N(M)\cap \Z_N$ so $\ker(f_2)\subset \st_N(M)\cap \Z_N$ and $f$ is a compatible central extension.
		
		If $f$ is a compatible central extension, then we know that $\ker(f_1)\subset M^N$ and $\ker(f_2)\subset \st_N(M)\cap\Z(N)$. We need to prove that $\ker(f_2)\subset \Z_B(N)$ so is enough to show that $\Z_B(N)=\Z(N)\cap\st_N(M)$. Since $Z_B(N)\subset \st_N(M)\cap \Z(N)$ we need to prove that $\st_N(M)\cap \Z(N)\subset Z_B(N)$.
		
		For that we will take $n\in \st_N(M)\cap \Z(N)$. We need an arbitrary element of $N=[N,N]$ so we will take $x=[n_1,n_2]\in N$. We have:
		\begin{align*}
		&\{n,x\}=\{n,[n_1,n_2]\}=\{n,\partial\{n_1,n_2\}\}=n\cdot \{n_1,n_2\}=0,\\
		&\{x,n\}=\{[n_1,n_2],n\}=\{\partial\{n_1,n_2\},n\}=-n\cdot\{n_1,n_2\}=0,
		\end{align*}
		where we use that $n\in \st_N(M)$.
		
		We have then that $\ker(f_2)\subset \Z_B(N)$ and $f$ is a central extension.
	\end{proof}

\begin{prop}
	Let $\mathcal{X}=(M\xrightarrow{\partial}N,\cdot,\{-,-\})$ is braided crossed module of Lie $K$-algebras.
	
	If $(M\xrightarrow{\partial}N,\cdot,\{-,-\})$ is a perfect braided crossed module (using Lemma~\ref{perfissame} is the same to say that $(M\xrightarrow{\partial}N,\cdot)$ is perfect), then its universal central extension $\mathcal{U}\xrightarrow{\Phi}\mathcal{X}$ (Corollary~\ref{ifperhaveuniv}) and its universal compatible central extension $\mathcal{V}\xrightarrow{c}\mathcal{X}$ (Proposition~\ref{univcompcentrext}) are isomorphic with the unique $h\colon\mathcal{U}\xrightarrow{}\mathcal{V}$ such that $\Phi=c\circ h$.
\end{prop}
	\begin{proof}
		Since $\mathcal{V}\xrightarrow{c}\mathcal{X}$ is a compatible central extension, we know using Lemma~\ref{perfissamecenter} (by hypothesis $(M\xrightarrow{\partial}N,\cdot,\{-,-\})$ is perfect) that it is a central extension, so using the universality of $\mathcal{U}$ we know that there are an unique morphism $\mathcal{U}\xrightarrow{h}\mathcal{V}$ such that $\Phi=c\circ h$.
		
		Since $\mathcal{U}\xrightarrow{\Phi}\mathcal{X}$ is a central extension is immediate that it is a compatible central extension $(\Z_B(N)\subset \st_N(M)\cap \Z(N))$ so by universality of $\mathcal{V}$ we have that exists an unique morphism $\mathcal{V}\xrightarrow{h'}\mathcal{U}$ such that $c=\Phi\circ h'$.
		
		Using the universality of $\mathcal{U}$, since $\Phi\circ (h'\circ  h)=c\circ h=\Phi$, we know that $h'\circ h=\Id_{\mathcal{U}}$.
		
		By the same arguments using the universality of $\mathcal{V}$, we have that $h\circ h'=\Id_{\mathcal{V}}$.
		
		So we have that $h$ is an isomorphism between $\mathcal{U}$ and $\mathcal{V}$ with inverse $h'$.
	\end{proof}

\begin{cor}
	If $M$ is a perfect Lie $K$-algebra, then $M\otimes (M\otimes M)\simeq M\otimes M$.
\end{cor}

	\begin{proof}
		If $M$ is perfect, then $(M\otimes M\xrightarrow{\partial}M,\cdot,\{-,-\})$ is a perfect braiding crossed module with $\partial(m\otimes m')=[m,m']$, $m\cdot (m_1\otimes m_2)=m\otimes [m_1,m_2]$ and $\{m_1,m_2\}=m_1\otimes m_2$. This is trivial since $M\otimes M$ is generated by $m_1\otimes m_2=\{m_1,m_2\}$.
		
		It has a universal central extension $(M\otimes M\xrightarrow{\Id_{M\otimes M}}M\otimes M,[-,-],[-,-])$ and a universal compatible central extension $(M\otimes (M\otimes M)\xrightarrow{\Id_M\otimes\partial}M\otimes M,*,\lBrace -,-\rBrace)$, where $(m_1\otimes m_2)*(m\otimes (m'\otimes m''))=[[m_1,m_2],m]\otimes (m'\otimes m'')+m\otimes ([m_1,m_2]\otimes [m', m''])$ and $\lBrace m_1\otimes m_2,m_3\otimes m_4\rBrace=[m_1,m_2]\otimes (m_3\otimes m_4)$.
		
		We know that both are isomorphic with $h$ with inverse $h'$.
		
		Then $h_1\colon M\otimes M\xrightarrow{}M\otimes(M\otimes M)$ is an isomorphisms of Lie $K$-algebras with inverse $h'$.
		
		Using the previous demonstration we know that the morphisms are the following ones:
		
		Note that $\Phi_1=\Id_{M\otimes M}$ since $\{-,-\}=-\otimes -$, then $\Phi_1(m\otimes m')=\{m,m'\}=m\otimes m'$.
		
		We have that $h_1(m\otimes m')=\lBrace \overline{m},\overline{m'}\rBrace$ where $c_2(\overline{n})=n$ for $n\in M$. But if $n\in M$ since $M$ is perfect, we know that $n=[n_1,n_2]$ and we can take $\overline{n}$ as $n_1\otimes n_2$ since $c_2(n_1\otimes n_2)=[n_1,n_2]=n$.
		
		Using this notation, we have that if $m\otimes m'\in M\otimes M$, then:
		\begin{align*}
		&h_1(m\otimes m')=\lBrace\overline{m},\overline{m'}\rBrace=\lBrace m_1\otimes m_2,m'_1\otimes m'_2\rBrace\\
		&=[m_1,m_2]\otimes \{m'_1,m'_2\}=m\otimes (m'_1\otimes m'_2).
		\end{align*}
		Where $m'_1, m'_2\in M$ verifies $[m'_1,m'_2]=m'$.
		
		Since $\Phi_1=\Id_{M\otimes M}$ we have that $\Phi_1=c_1\circ h_1$. Since we know that $h_1$ is invertible that $h'_1=c_1$. This is coherent since in the other hand we have $c=\Phi\circ h'$, so we have $c_1=\Psi_1\circ h'=\Id_{M\otimes M}\circ h'=h'$.
		
		For that we only need to find which is the expression of $c_1$ to check which is the inverse of the isomorphism $h$.
		
		For $m\otimes (m'\otimes m'')\in M\otimes (M\otimes M)$ we have:
		\begin{equation*}
		c_1(m\otimes (m'\otimes m''))=m\cdot (m'\otimes m'')=m\otimes [m',m''].
		\end{equation*}
		
		We can check now if they are inverses one from another:
		\begin{equation*}
		c_1\circ h_1(m\otimes m')=c_1(m\otimes (m'_1\otimes m'_2))=m\otimes [m'_1,m'_2]=m\otimes m'.
		\end{equation*}
		\begin{equation*}
		h_1\circ c_1(m\otimes (m'\otimes m''))=h_1(m\otimes[m',m''])=m\otimes ([m',m'']_1,[m',m'']_2).
		\end{equation*}
		Where $[m',m'']_i\in M$ for $i\in \{1,2\}$ and $[[m',m'']_1,[m',m'']_2]=[m',m'']$. For this last equality we can take $[m',m'']_1=m$ and $[m',m'']_2=m''$ as images by $c_2$. Then:
		\begin{equation*}
		h_1\circ c_1(m\otimes (m'\otimes m''))=m\otimes ([m',m'']_1,[m',m'']_2)=m\otimes (m'\otimes m'').
		\end{equation*}
		So we have explicitly the isomorphism.
	\end{proof}

\end{document}